\documentclass[12pt]{amsart}
\usepackage{amsfonts}
\usepackage{amsmath}
\usepackage{amssymb}
\usepackage{amsthm}
\usepackage{color, ulem}
\usepackage{tikz-cd}
\usepackage{enumitem}
\usepackage{combelow}

\makeatletter
\def\author@andify{%
  \nxandlist {\unskip ,\penalty-1 \space\ignorespaces}%
    {\unskip {} \@@and~}%
    {\unskip ,\penalty-2 \space }%
}
\makeatother

\input xy
\xyoption {all}
\def \I {{\mathcal I}}

\def \P {{\mathbb P}}
\def \M {{\mathsf M}}
\def \e {{\mathcal E}}
\def \f {{\mathcal F}}
\def \S {{\mathsf S}}
\def \ext {{\mathcal {E}xt}}

\DeclareMathOperator{\ch}{ch}

\DeclareMathOperator{\sm}{small}

\newcommand{\comment}[1]{}
\renewcommand{\tilde}[1]{\widetilde{#1}}
\renewcommand{\bar}[1]{\overline{#1}}
\newtheorem{theorem}{Theorem}
\newtheorem{lemma}{Lemma}
\newtheorem{corollary}{Corollary}

\newtheorem{conjecture}{Conjecture}

\theoremstyle{definition}
 
\theoremstyle{definition}
\newtheorem{remark}{Remark}
\theoremstyle{definition}
\newtheorem{definition}{Definition} 

\newtheoremstyle{TheoremNum}
        {7pt}{7pt}              
        {\itshape}                      
        {}                              
        {\bfseries}                     
        {.}                             
        { }                             
        {\thmname{#1}\thmnote{ \bfseries #3}}
    \theoremstyle{TheoremNum}
    \newtheorem{thmn}{Theorem}
    \newtheorem{conj}{Conjecture}

\usepackage{hyperref}

\begin{document}

\baselineskip=15pt
\begin{center}
\title [On product identities] {On product identities and the Chow rings of holomorphic symplectic varieties}

\author[I. Barros]{Ignacio Barros}
\author[L. Flapan]{Laure Flapan}
\author[A. Marian]{Alina Marian}
\author[R. Silversmith]{Rob Silversmith}

\address[I. Barros]{
Universit\'{e} Paris-Saclay, LMO\\
Rue Michel Magat, B\^{a}t. 307,
91405 Orsay, France} 
\email{ignacio.barros@universite-paris-saclay.fr}

\address[L. Flapan]{
Department of Mathematics\\
Massachusetts Institute of Technology\\
77 Massachusetts Ave.\\
Cambridge, MA 02139, USA} 
\email{lflapan@mit.edu}

\address[A. Marian]{
Department of Mathematics\\
Northeastern University\\
360 Huntington Ave.\\
Boston, MA 02115, USA} 
\email{a.marian@northeastern.edu}

\address[R. Silversmith]{
Department of Mathematics\\
Northeastern University\\
360 Huntington Ave.\\
Boston, MA 02115, USA} 
\email{r.silversmith@northeastern.edu}

\begin{abstract}
For a moduli space $\M$ of stable sheaves over a $K3$ surface $X$, we propose a series of conjectural identities in the Chow rings $CH_\star (\M \times X^\ell),\, \ell \geq 1,$ generalizing the classic Beauville-Voisin identity for a $K3$ surface. We emphasize consequences of the conjecture for the structure of the tautological subring $R_\star (\M) \subset CH_\star (\M).$ The conjecture places all tautological classes in the lowest piece of a natural filtration emerging on $CH_\star (\M)$, which we also discuss. We prove the proposed identities when $\M$ is the Hilbert scheme of points on a $K3$ surface.
\end{abstract}

\comment{
\begin{abstract}
We propose conjectural identities in the Chow rings of hyperk\"{a}hler varieties of $K3$ type which generalize the classical Beauville-Voisin identity for $K3$ surfaces. We emphasize consequences of the conjecture for the structure of the tautological subring of the Chow ring. We prove the proposed identities for the Hilbert scheme of points on a K3 surface.

\end{abstract}}
\maketitle
\end{center}

\section{Introduction}\label{sec: intro}

\vskip.3in

Understanding the Chow ring of irreducible holomorphic symplectic varieties is a problem of considerable interest. In the case of a
smooth projective $K3$ surface $X$, an essential role in approaching the cycle structure is played by a distinguished zero-cycle $c_X$, first noted and studied in \cite{BV}.
The cycle $c_X$ has degree one and is the Chow class of any point lying on a rational curve in $X$. The intersection of any two divisors is a multiple
of $c_X$, while the second Chern class of the tangent bundle satisfies $$c_2 (TX) = 24 \,c_X.$$

In higher dimensions, it is natural to consider moduli spaces of stable sheaves on $K3$ surfaces. For a smooth projective $K3$ surface $X$, we let $v \in H^{\star} (X, \, \mathbb Z)$ be a primitive Mukai vector, and let $\M$ be the moduli space of stable sheaves over $X$ with  Mukai vector $v$, relative to a $v$-generic polarization. We note nevertheless that all statements in this text apply more broadly to moduli spaces of Bridgeland-stable sheaf complexes with respect to a $v$-generic stability condition $\sigma$.
The moduli space $\M$ is a smooth projective irreducible holomorphic symplectic variety (cf. \cite{huybrechts2}, Section 10.2 and the essential references therein) of dimension
$$\dim  \M = m = \langle v, \, v \rangle +2,$$
admitting a quasi-universal sheaf  $$\mathcal F \to \M \times X.$$ To keep the exposition simple we assume in fact that $\M$ is a fine moduli  space, so $\f$ is a universal object. The restriction is not essential, as we later explain. We denote the two projections by $\pi: \M \times X \to \M$ and $\rho: \M \times X \to X.$

In parallel with the $K3$ geometry, there is a distinguished zero-cycle $$c_{\M} \, \in \, CH_0 (\M)$$ of degree one: this is the class of any stable sheaf $F$ such that 
\begin{equation}
\label{special}
c_2(F) = k \,c_X \, \, \, \text{in} \, \, \, CH_0 (X),
\end{equation}
where $k$ is the degree of the second Chern class specified by the Mukai vector $v$. Sheaves satisfying \eqref{special} exist in $\M$ (cf. \cite{OG2}), and have the same Chow class as shown in \cite{MZ}, following a conjecture of \cite{SYZ}. In analogy with the surface situation, one expects (\cite{V1}, \cite{SYZ}) that the special cycle corresponds to the largest rational equivalence orbit of points on $\M$. The intersection-theoretic properties of $c_{\M}$ are not understood as well as those of its counterpart $c_X$ in the two-dimensional context.  

We study the geometry of the universal sheaf and of the special cycles $c_X$ and $c_{\M}$ in two strands:
\vskip.1in

\begin{enumerate}[label=(\roman*)]
\item
We single out the {\it tautological subring} $R_{\star} (\M) \subset CH_{\star} (\M)$, generated by the classes
$$\pi_{\star} \left (M (c_i (\f)) \cdot \rho^{\star} \beta \right ),$$
with $M$ any monomial in the Chern classes of $\f$, and $\beta$ any class in the Beauville-Voisin subring
$$R_\star (X) = CH_2 (X) + CH_1 (X) + \mathbb Z \,c_X \subset CH_\star (X).$$

\comment{\begin{itemize}
\item $\pi_{\star} (M (c_i (\f))),$ with $M$ any monomial in the Chern classes of $\f$;
\item $\pi_{\star} \left (M (c_i (\f)) \cdot \rho^{\star} D \right ),$ with $D \in CH_1 (X), \, M$ any monomial in the Chern classes of $\f$;
\item $\pi_{\star}\left  (M (c_i (\f)) \cdot \rho^{\star} c_X \right ),$ with $M$ any monomial in the Chern classes of $\f$.
\end{itemize}}

\item
We emphasize the rank zero virtual sheaf
$$\overline \f = \f - \rho^{\star} F, \, \, \, \text{with} \, \, \, F \in \M \, \, \text{such that}  \, \, [F] = c_{\M} \in CH_0 (\M).$$
Intuitively, the second Chern class of $\overline\f$ reflects to some extent the variation of rational equivalence classes across points in $\M$, relative to the special class. 
\end{enumerate}

\medskip

These two strands come together naturally within the framework of the following, which is the main conjecture of the paper.

\begin{conjecture}
\label{principal}
Let $\alpha \in R_{\star} (\M)$ be a tautological class of codimension $d$. Consider the product $\M \times X^{\ell}$, where $d + \ell > \dim \M.$ Let $\overline \f_i$ denote the 
pullback to $\M \times X^{\ell}$ of the virtual universal sheaf on $\M$ and the $i$th factor of $X$. Then for every $i_1, \ldots, i_\ell \geq 0,$ 
\begin{equation}
\label{main}
\alpha \cdot ch_{i_1} (\overline \f_1) \cdots ch_{i_\ell} (\overline \f_{\ell}) = 0 \, \, \, \in \, \, \, CH_{\star} (\M \times X^{\ell}).
\end{equation}
\end{conjecture}
The K\"{u}nneth components along $\M$ of the Chern classes $c_i (\overline \f)$ on $\M \times X$ have positive cohomological degrees for $i>0$. Since $\M$ has no odd cohomology (in particular $b_1(\M)=0$), the products \eqref{main} are homologically trivial for dimension reasons due to the inequality $d + \ell > \dim \M$.

Conjecture \ref{principal} yields a rich collection of interesting Chow identities and we highlight some of them now. In case $\M = X,$ viewed trivially as the Hilbert scheme of one point on itself, we have
$$\overline \f = \overline \I, \, \, \, \text{where} \, \, \, \overline {\I} = \I_{\Delta} - \I_{X \times c} \, \, \, \text{on} \, \, \, X \times X,$$
with $c \in X$ a point of Chow class $c_X$.
Therefore $$ch_2 (\overline \f) = ch_2 (\overline \I) = - \, \overline \Delta,$$ where we have set $$ \overline \Delta =  \Delta - X \times c_X  \, \, \, \, \text{in} \, \, \, \,  CH_2 (X \times X). $$
Thus when $\M = X$, for the tautological class $\alpha = 1,$ the identity $$ch_2 (\overline \f_1) \cdot ch_2 (\overline \f_2) \cdot ch_2 (\overline \f_3) = 0 \, \, \, \text{in} \, \, \, CH_2 (X \times X^3)$$ predicted under \eqref{main} takes the form 
\begin{equation}
\label{bv}
\overline \Delta_{01} \cdot \overline \Delta_{02} \cdot \overline \Delta_{03} = 0 \, \, \, \text{in} \, \, \,  CH_2 (X \times X^3),
\end{equation}
while the $K$-theoretic identity 
\begin{equation} 
\label{ideal}
\overline \I_{01} \cdot \overline \I_{02} \cdot \overline \I_{03} = 0 \, \, \, \text{in} \, \, \, K (X \times X^3)
\end{equation}
also holds. Here the index 0 is used to keep track of the first distinguished factor $X$ in the quadruple product $X^4$, and $\overline \Delta_{0i}$ and $\overline \I_{0i}$ indicate pullbacks from the 0th and $i$th factors. 

Pushing forward to the product of the last three factors, equation \eqref {bv} is easily seen to be equivalent to the fundamental Beauville-Voisin identity \cite{BV}
\begin{equation}
\label{bv0}
\Delta - \Delta_{c} + \Delta_{c,  c} = 0 \, \, \, \text{in} \, \, \, CH_2 (X \times X \times X).
\end{equation}
Here $c$ again denotes a fixed point of Chow class $c_X$; $\Delta$ is the small diagonal of points $(x,x,x)$; $\Delta_c$ consists of triples of the form $(x, x, c), \, (c, x, x), \, (x, c, x)$; $\Delta_{c,c}$ is the set of triples of the form $(c, c, x), \, (c, x, c), \, (x,c,c)$ for $x \in X.$

If we now take $\alpha = D,$ a divisor class on $X$, the vanishing $$\alpha \cdot ch_2 (\overline \f_1) \cdot ch_2 (\overline \f_2) = 0 \, \, \, \text{in} \, \, \, CH_1 (X \times X^2)$$ predicted by Conjecture \ref{principal} becomes
\begin{equation}
\label{second}
D^{(0)}\cdot \overline \Delta_{01} \cdot \overline \Delta_{02} = 0 \, \, \, \text{in} \, \, \,  CH_1 (X \times X^2),
\end{equation}
(with the divisor $D$ pulled back from the $0$th factor) which is known to hold. Indeed, any divisor class $D$ on $X$ is a linear combination of classes of rational curves on $X$, and the cycles $\overline \Delta_{01}, \, \overline \Delta_{02}$ restrict to zero for every point on a rational curve in the $0$th factor $X$. The spreading principle (cf. \cite[Theorem 3.1]{V2}, see also Remark \ref{Lagrangian} in Section \ref{sec: taut} below) implies \eqref{second}.
We also note the trivial identity 
\begin{equation}
\label{third}
c_X^{(0)} \cdot \overline \Delta_{01} = 0 \, \, \, \text{in} \, \, \,  CH_0 (X \times X).
\end{equation}

Returning now to the case of a general moduli space $\M$, we see that for $\alpha = 1,$ the expected vanishing
\begin{equation}
ch_2 (\overline \f_1) \cdot ch_2 (\overline \f_2) \cdots ch_2 (\overline \f_{m+1}) = 0 \, \, \, \text{in} \, \,\,  CH_m (\M \times X^{m+1}) 
\end{equation}
predicted by Conjecture \ref{principal} is the natural generalization of the Beauville-Voisin fundamental identity \eqref{bv0} in the triple product of a $K3$ surface. 
The beautiful identity 
\begin{equation}\label{conjKTheory}
\overline \f_1 \cdot \overline \f_2 \cdots \overline \f_{m+1} = 0 \, \, \, \text{in} \, \, \, K (\M \times X^{m+1})
\end{equation}
is also predicted by Conjecture \ref{principal}.

For tautological classes $\alpha \in R_{\star} (\M)$ of positive codimension, the series of identities predicted by Conjecture \ref{principal} should be viewed as generalizing  \eqref{second} and \eqref{third} from the $K3$ context to a general moduli setup. 

\vskip.1in

The identities of Conjecture 1 lead in turn to a large collection of conjectural Chow vanishings in the self-products $\M \times \M \times \cdots \times \M.$ 
We set 
$$\overline \Delta = \Delta - \M \times c_{\M} \, \, \, \text{in} \, \, \, CH_m (\M \times \M),$$
and observe
\begin{theorem}
\label{thm:applications}
The system of identities \eqref{main} of Conjecture \ref{principal} is equivalent to the vanishing
\begin{equation}
\label{mconseq}
\alpha \cdot \overline \Delta_{01} \cdots \overline \Delta_{0, \ell} = 0 \, \, \, \text{in} \, \, \, CH_{\star} (\M \times \,\M^{\ell}).
\end{equation}
for any tautological class $\alpha \in R_{\star} (\M)$ of codimension $d$ and integer $\ell$ satisfying $$d + \ell > \dim \M.$$ Here the first factor of $\M$ is labeled by 0, and $\alpha$ is pulled back from this factor. 
\end{theorem}

Theorem \ref{thm:applications} is immediately seen to have a few interesting consequences. 
At one end, if we take $\alpha \in R_{\star} (\M)$ to be a tautological zero-cycle of degree $a$, and pick $\ell=1,$ we obtain the vanishing
$$\alpha \cdot \overline \Delta_{01}  = 0 \, \, \,\text{in} \,\, \, CH_0 (\M \, \times \, \M).$$
Pushing forward to the second factor of $\M$, this gives $$\alpha = a \, c_{\M},$$ and allows us to conclude
\begin{corollary}\label{cor: 0 cycle}
Assuming Conjecture \ref{principal} holds, the tautological ring $R_{\star} (\M) $ has rank one in dimension zero:
$$R_0 (\M) = \mathbb Q \cdot c_{\M}.$$
\end{corollary}

At the other end, taking $\alpha = 1$ yields the vanishing
\begin{equation}
\label{moddiag}
{\overline {\Delta}}_{01} \cdots \overline \Delta_{0,m+1} = 0 \, \, \, \text{in} \, \, \, CH_m (\M\, \times \, \M^{m+1}).
\end{equation}
It is easy to see that the pushforward of this product cycle, via the projection $\M \, \times \, \M^{m+1}\to \M^{m+1}$ forgetting the first factor, is the modified diagonal cycle studied in \cite{OG3}. Conjecture \ref{principal} recovers in this case the vanishing of the modified diagonal conjectured in \cite{OG3}.
\begin{corollary}\label{cor:modified diag}
Assuming Conjecture \ref{principal} holds, the modified diagonal cycle 
$$\Gamma^{m+1} (\M, \, c_{\M}) = \Delta - \Delta_c + \Delta_{c,c} - \cdots + \Delta_{c,c,\ldots, c}$$ vanishes in 
$CH_m (\M^{m+1}).$
\end{corollary}

It is natural to consider, for every codimension $0 \leq d \leq m,$ the increasing filtration $$\S_0 \subset \S_1 \subset \ldots \S_i \subset \ldots \subset CH^d (\M),$$ given by
 $$\S_i (CH^d (\M)) =  \{ \alpha \, \, \text{with} \, \, \alpha \cdot \overline \Delta_{01} \cdots \overline \Delta_{0, \,m-d+i+1} = 0 \subset CH_{\star} (\M \times \,\M^{m-d+i+1})\}.$$
Here $\alpha$ is pulled back to the product from the first factor. While the filtration does terminate (see \cite{V3}, Corollary 1.6 and Proposition 2.2), the conjectured vanishing \eqref{moddiag} would further establish
$$\S_d (CH^d (\M)) =  \{ \alpha  \, \, \text{with} \, \, \alpha \cdot \overline \Delta_{01} \cdots \overline \Delta_{0, \,m+1} = 0 \subset CH_{\star} (\M \times \,\M^{m+1})\} = CH^d (\M).$$
Aspects of the $\S_\bullet$ filtration are discussed below in Section \ref{split_filtration}. We only notice here fleetingly
that 
$$\S_0 (CH^d (\M)) = \{ \alpha \, \, \text{with} \, \, \alpha \cdot \overline \Delta_{01} \cdots \overline \Delta_{0, \,m-d+1} = 0 \subset CH_{\star} (\M \times \,\M^{m-d+1})\}.$$
Thus for every $d$, Conjecture \ref{principal} and Theorem \ref{thm:applications} would place all codimension $d$ tautological classes in $\S_0 (CH^d (\M)).$ It is also known (cf. \cite{SYZ}) that $\M$ admits Lagrangian constant-cycle subvarieties for the special cycle $c_\M$. 
We explain later, in Remark \ref{Lagrangian} of Section \ref{sec: taut}, that the spreading principle places these subvarieties in $\S_0 (CH^n (\M))$, for $n=m/2.$ We further conjecture there that Chow classes of Lagrangian constant-cycle subvarieties are in fact tautological.

\vskip.1in

As evidence for Conjecture \ref{principal}, we show:

\begin{theorem}
\label{principal0} 
Let $X$ be a smooth projective $K3$ surface. Conjecture \ref{principal} holds for $\M = X^{[n]},$ the Hilbert scheme of $n$ points on $X$. 
\end{theorem}

A richer version of Theorem \ref{thm:applications} is proven in Section \ref{sec: taut} as Theorem \ref{thm:applications}$^{*}.$ 
\comment{Indeed the vanishings \eqref{mconseq} are shown to belong to a larger collection of Chow identities in the products $\M \times \cdots \times  \M$ implied by Conjecture \ref{principal}.}  
Section \ref{sec: taut} also discusses the tautological subring of the Chow ring in a broader context, for the products $\M \times X^{\ell}, \, \, \ell \geq 0.$ This leads to generalizations of Conjecture \ref{principal} and Theorem \ref{principal0} to the setting of a product $\M \times X^{\ell}$ which are needed in the inductive argument of Section \ref{sec: hilb scheme}.

Theorem \ref{principal0} is argued in Section \ref{sec: hilb scheme} inductively on the number of points, using the geometry of the nested Hilbert scheme. The inductive technique, introduced in \cite{EGL}, is well understood in the context of universality arguments for series of intersection-theoretic invariants on the Hilbert scheme of a surface. In Theorem \ref{principal0},  we pursue nevertheless {\it Chow} identities in a large range of codimensions. Beyond the standard elements of the \cite{EGL} mechanism, we accordingly make strong use of the overhaul, on the level of Chow groups, of the Nakajima-Lehn commutator identities previously known to hold in cohomology. This overhaul was recently completed in \cite{MN}. 
Finally, the base case of the induction beautifully comes down to the three fundamental $K3$ identities \eqref{bv}, \eqref{second} and \eqref{third}. 

Together, Theorems \ref{thm:applications} and \ref{principal0} establish unconditionally a hierarchy of Chow identities involving the full tautological ring of the Hilbert scheme of a $K3$ surface, for all codimensions. In this hierarchy, the ``codimension-zero'' identity given by Corollary \ref{cor:modified diag} is the vanishing of the modified diagonal cycle first established in \cite{V3} (conditional on the forthcoming \cite{V6}) which is thus placed in a vaster framework. The top codimension case is the assertion that the special cycle spans the tautological ring in dimension zero. 

In Section \ref{split_filtration} we discuss the $\S_\bullet$ filtration on all Chow groups introduced above. For zero-cycles, we compare it to the filtrations on $CH_0 (\M)$ studied in \cite{V1}, \cite{SYZ}. In higher dimensions, the placement of constant-cycle subvarieties within the filtration is investigated. 

\vskip.1in

In the context of the tautological ring (see Definition \ref{def: taut ring}) of the product $\M \times X^{\ell}$,  it is natural to formulate the stronger 
\begin{conjecture}
\label{bvinject}
The restriction of the cycle class map to the tautological subring, $$\tau: R_{\star} (\M \times X^{\ell}) \to H_{\star} (\M \times X^{\ell}),$$ is injective. 
\end{conjecture}
This statement follows the line of conjectures on the injectivity of the cycle class map on suitable subrings of Chow initiated in \cite{beauville2}, \cite{V4}. It 
completely subsumes Conjecture \ref{principal}, our main object in this paper. Indeed the identities \eqref{main} are among tautological classes and hold in homology. The advantage of Conjecture \ref{principal} (and the equivalent statement given in Theorem \ref{thm:applications}) is that it proposes a {\it concrete} set of relations in the Chow ring, and is thus easier to come to grips with than the elusive Conjecture \ref{bvinject}. 

Finally, the interesting problem of understanding corrections of the identities \eqref{main} in a relative setting over the moduli space of polarized $K3$ surfaces is left for future exploration. 

\vskip.2in

{\it Addendum.}  We note that in hindsight Theorem \ref{principal0} may also be argued from the results of the contemporaneous paper \cite{noy}, which also relies on the Nakajima-Lehn identities established in \cite{MN}. As shown in \cite{noy}, the main operator $h$ which gives the Chow decomposition on $X^{[n]}$ is a derivation, and its action on the universal Chern character is explicitly determined. Together, these two ingredients can be shown to yield the vanishing of the product cycles of Conjecture \ref{principal} in the Hilbert scheme case.  

The ${\mathsf S}_\bullet$ filtration discussed here was also recently studied in \cite{vial}. For zero-cycles, \cite{vial} shows that it agrees with the filtration introduced in \cite{SYZ}, while only one inclusion is pointed out in the present article (cf. Lemma \ref{lem:FiltrationInclusion}). We thank Charles Vial for a useful correspondence on this topic.

\subsection*{Acknowledgements} We thank Mark de Cataldo, Lie Fu, Daniel Huybrechts, Eyal Markman, Davesh Maulik, Andrei Negu\cb{t}, Drago\cb{s} Oprea, Junliang Shen, Charles Vial, Qizheng Yin, and Ruxuan Zhang for helpful discussions and correspondence. Conversations with Andrei Negu\cb{t} regarding the Chow induction on number of points in the Hilbert scheme context played a particularly important role. I.B. is supported by the ERC Synergy Grant ERC-2020-SyG-854361-HyperK. L.F. was supported by the NSF through grant DMS 1803082. A.M. was supported by the NSF through grants DMS 1601605 and 1902310, as well as by the Radcliffe Institute for Advanced Study at Harvard, through a 2019-2020 Radcliffe Fellowship. R.S. was supported by the NSF through grant DMS 1645877. This project was initiated at Northeastern University under the auspices of the NSF-funded RTG grant DMS 1645877. 

\vskip.3in

\section{Tautological rings and product cycles }\label{sec: taut}

\vskip.3in

\subsection{Tautological rings}\label{sec: taut1}

Let $X$ be a smooth projective $K3$ surface, $v \in H^{\star} (X, \, \mathbb Z)$ a primitive Mukai vector, and let $\M$ be the moduli space of stable sheaves with Mukai vector $v$ on $X$ relative to a $v$-generic polarization. To ease the exposition, we make the nonrestrictive assumption that $\M$ admits a universal family $$\f \to \M \times X.$$ If this is not the case, we use instead the universal \textit{Chern character} $\text{ch} \, \mathcal F \in CH_{\star} (\M \times X)$ discussed in \cite{M}, \cite{markman2} to define the tautological ring, as well as to formulate the identities \eqref{main} of Conjecture \ref{principal}. The construction in Section 3 of \cite{M} and Section 3.1 of \cite{markman2} follows the original explanation of \cite{mukai2}, Appendix A.5. Specifically, a universal sheaf can be glued together over $\mathbb P \times X,$ where $\mathbb P \to \M$ is a  suitable projective bundle over $\M$. After appropriate twisting, the universal Chern character (but not the sheaf) is seen to descend as a rational class to the product $\M \times X$. When available, the universal sheaf is defined up to tensoring by line bundles from $\M$. If a universal sheaf is not available, the universal Chern character is correspondingly (cf. \cite{M}, Section 3) defined up to multiplication by the Chern character of a line bundle over the moduli space. We now recall from the introduction:
\begin{definition}\label{def: taut ring M}
The tautological ring $$R_{\star} (\M) \subset CH_{\star} (\M)$$ is the subring of Chow generated by the following classes
\begin{itemize}
\item $\pi_{\star} (M (c_i (\f))),$ \, $M$ any monomial in the Chern classes of $\f$;
\item $\pi_{\star} \left (M (c_i (\f)) \cdot \rho^{\star} D \right ),$ $D \in CH_1 (X), \, M$ any monomial in the Chern classes of $\f$;
\item $\pi_{\star}\left  (M (c_i (\f)) \cdot \rho^{\star} c_X \right ),$ $M$ any monomial in the Chern classes of $\f$.
\end{itemize}
\end{definition}

\noindent 
Thus $R_\star (\M)$ is generated by all classes of the form
$$\pi_\star (M(c_i (\f)) \cdot \rho^\star \beta ),$$ for an arbitrary monomial $M$ in the Chern classes of $\f$, and $\beta \in R_\star (X)$ any class in the Beauville-Voisin ring $$R_\star (X) = \mathbb Z \, c_X + CH_1 (X) + CH_2 (X) \subset CH_\star (X).$$
As we repeatedly consider Chern characters, we will work throughout with $\mathbb Q$ coefficients. 

It is useful to extend Definition \ref{def: taut ring M} to the arbitrary products $$\M \times X^{k}, \, \,\, \text{for} \, \, \, k \geq 0.$$ For $1 \leq s \leq k,$ let $\rho_s : \M \times X^k \to X$ be the map to the factor indexed by $s$, with the accompanying projection  $\bar \rho_s: \M \times X^k \to \M \times X.$
Denote by $$\f_s = \bar \rho_s^{\star} \f$$ the universal sheaf on 
$\M \times X^k$ pulled back from $\M \times X$ via the $s$th projection. 

\begin{definition}\label{def: taut ring}
The tautological system of rings $R_{\star} (\M \times X^k) \subset CH_{\star} (\M \times X^k), \, \, k \geq 0$ is the smallest system of $\mathbb Q$-algebras satisfying the following three properties:

\begin{enumerate} [label=(\roman*)]
\vskip.1in
\item
$R_{\star} (\M \times X^k)$ contains the Chern classes $c_i (\f_s), \, \, \, 1\leq s \leq k$, as well as the classes $\rho_s^*D$, for $D \in CH_1 (X).$

\vskip.1in

\item 
The system is closed under pushforward via the natural projections $\pi: \M \times X^n \to \M \times X^k$ forgetting factors of $X$, where $n \geq k.$ 

\vskip.1in

\item 
The system is closed under pushforward via the natural inclusions $\iota: \M \times X^n \to \M \times X^k$ for $n \leq k$ through diagonal embeddings of factors of $X$ or embeddings using the special cycle $c_X$. 
\end{enumerate}
\end{definition}

Concretely, this means that for each $k \geq 1,$ the subring $R_{\star} (\M \times X^k) \subset CH_{\star} (\M \times X^k)$ is generated by the following classes:

\begin{itemize}
\vskip.1in
\item the pullbacks $\pi^{\star} \alpha$ from $\M$ to the product $\M \times X^k$, where $\alpha \in CH_{\star}( \M)$ is tautological in the sense of Definition 1.  
\vskip.1in
\item the Chern classes $c_i (\f_s), \, \, \, 1\leq s \leq k$;
\vskip.1in
\item the pullback classes $\rho_s^{\star} D, \, \, \, \rho_s^{\star} c_X, \, \, 1\leq s \leq k$ ;
\vskip.1in
\item the diagonal classes $\rho_{rs}^{\star} \Delta, \, \, \, 1 \leq r,s \leq k.$
\end{itemize}

It is indeed straightforward to check that each of the classes above is in $R_\star(\M\times X^k),$ and that any polynomial in these classes satisfies the three properties in Definition \ref{def: taut ring}. Thus they generate $R_\star(\M\times X^k).$

\vskip.1in

We will further let $$R_* (X^k) \subset CH_\star (X^k)$$ be the {\it Beauville-Voisin subring} generated by the pullback classes $\rho_s^{\star} D, \, \, \, \rho_s^{\star} c_X, \, \, 1\leq s \leq k,$ and the diagonal classes $\rho_{rs}^{\star} \Delta, \, \, \, 1 \leq r,s \leq k.$

\begin{remark}
It is interesting to examine the independence of $R_\star (\M)$ of the modular interpretation of the holomorphic symplectic manifold $\M$. Suppose $$\M = \M_v \simeq \M_{v'},$$ where $\M_{v'}$ is a moduli space of stable sheaves with Mukai vector $v'$ relative to a polarization $H'$ over a $K3$ surface $X'$. There is a derived (anti-)equivalence \cite{bayermacri} $$\Phi: D^b (X) \simeq D^b (X')$$ with kernel $\e \in D^b (X \times X')$ inducing the isomorphism $$\overline \Phi: \M_v \to \M_{v'}.$$ Let $\f \to \M_v \times X, \, \f' \to \M_{v'} \times X'$ be the universal objects. Considering the extended equivalence induced by $\e$,  $$\Phi_\M: D^b (\M_v \times X) \simeq D^b (\M_{v} \times X'),$$ and using the identification $\overline \Phi$, we 
have, up to tensoring by a line bundle from $\M$,  
$$ \Phi_{\M} (\f) = \f' \, \, \, \text{on} \,\, \,  \M \times X'.$$
For $k \geq 1$, let now $\pi: \M \times X^k \to \M,\, \pi': \M \times (X')^k \to \M $ be the projections, and let 
$$\zeta = \pi'_\star \left (\text{ch}_{i_1} \f'_1\,  \cdots \, \text{ch}_{i_k} \f'_k  \cdot \beta' \right ), \, \, \, \beta \in R_\star ((X')^k),$$ be a class on 
 $\M$ tautological in the sense of $v'$.  As $$\text{ch} \, \f' = \text{ch} \, \Phi_{\M} (\f) = (\pi \times \text{id}_{X'})_\star \left ( \text{ch} \, \f \cdot \text{ch} \, \e \cdot \text{td} \, X \right ),$$ 
it is standard to write $\zeta$ as a linear combination of terms of type
\begin{equation}
\label{BVinvariance}
\pi_\star \left (\text{ch}_{j_1} \f_1\,  \cdots \, \text{ch}_{j_k} \f_k  \cdot \beta \right ), \, \, \, \beta \in CH_\star (X^k).
\end{equation}
If in each summand we had $\beta \in R_\star (X^k)$, the class $
\zeta$ would be tautological in the sense of $v$ as well. It is well known now by a theorem of Huybrechts (cf. \cite{huybrechts1}) that the equivalence $\Phi$ preserves the Beauville-Voisin ring on the Chow level. If in turn, for any $k \geq 2$, the derived equivalence $$\Phi^k: D^b (X^k) \to D^b ((X')^k) \, \, \, \text{with kernel} \, \, \,  \mathcal E^{\boxtimes k} \to (X\times X')^k$$ sent $R_* (X^k)$ to $R_\star ((X')^k)$ on the Chow level, then the classes $\beta$ occurring in \eqref{BVinvariance} would indeed be tautological. Relatedly, this raises the interesting question whether tautological ring invariance holds for any pair of derived-equivalent moduli spaces of sheaves over K3 surfaces. This would be a natural generalization of Huybrechts's theorem in dimension two. It would be worthwhile to investigate this circle of ideas further.

{}

\end{remark}
\vskip.1in

\subsection{Examples of tautological classes}

\subsubsection{Divisors} It is well-known (cf. \cite{mukai1} , \cite{mukai2}, \cite{OG1}, \cite{Y})  that the determinant line bundle homomorphism  
$$\Theta_v : v^{\perp} \to NS (\M)$$ is an isomorphism for $\langle v, \, v \rangle >0$, and is in all cases surjective. Here $$v^{\perp} \subset H^{\star}_{alg} (X, \, \mathbb Z)$$ denotes the orthogonal complement of the Mukai vector $v$ in the algebraic Mukai lattice $$H^{\star}_{alg} (X, \, \mathbb Z)=H^0(X,\mathbb{Z})\oplus NS(X) \oplus H^4(X,\mathbb{Z})$$ equipped with the Mukai pairing $$\langle(r,c,s),(r',c',s')\rangle=cc'-rs'-sr'\in \mathbb{Z}.$$ Divisors on $\M$ are thus tautological, making Definition \ref{def: taut ring M} independent of the choice of universal family/universal Chern character.

\subsubsection{Chern classes of the tangent bundle} We have $$\text{ch} \, (T\M) = 2 - \text{ch} \, \mathcal Ext_{\pi}^{\bullet} (\mathcal F, \, \mathcal F) = 2 - \pi_\star (\text{ch} \,\f^{\vee} \cdot \text{ch}\, \f \cdot \rho^\star \text{td}\, X),$$ therefore the Chern classes $c_i (T\M)$ are tautological.

\subsubsection{The special cycle $c_\M$}
We show now that the distinguished zero-cycle $c_{\M}$ is tautological.  
Consider the product $\M \times \M \times X$, equipped with the universal sheaves $\mathcal E, \, \mathcal F$ which correspond to the two copies of $\M$, 
and are pulled back to the product. Let $\pi: \M \times \M \times X \to \M \times \M$ be the projection. We form
the natural relative Ext complex (shifted for convenience),
$$\mathsf W (\mathcal E, \mathcal F) = \ext_{\pi}^{\bullet} (\mathcal E, \, \mathcal F)  [1] \, \, \, \text{on} \, \, \, \M \times \M.$$
We further fix $F_0 \to X$ a sheaf parameterized by $\M$, and denote by
$$\mathsf W (\mathcal E, \, \rho^{\star} F_0) = \ext_{\pi}^{\bullet} (\mathcal E, \,  \rho^{\star} F_0)  [1] \, \, \, \text{on} \, \, \, \M,$$
the pullback of $\mathsf W (\e, \, \f)$ under the inclusion $\M \times [F_0] \hookrightarrow \M \times \M.$

As observed in \cite{MZ}, the complex $\mathsf W$ plays a role in understanding Chow classes of points on $\M$, since the middle Chern class of $\mathsf W$ is the 
class of the diagonal in the product $\M \times \M$. The formula
\begin{equation}
\label{diag}
c_m \left (\mathsf W (\mathcal E, \mathcal F)\right ) = [\Delta] \, \, \, \text{in} \, \, \, CH_m (\M \times \M)
\end{equation}
was established in 
\cite{M}, and is aligned with earlier work of Beauville \cite{B} and Ellingsrud-Str{\o}mme \cite{ES} on representing the diagonal in terms of the 
universal Chern classes, in the context of moduli spaces of Gieseker-stable sheaves. 
By pullback, the diagonal formula \eqref{diag} gives 
\begin{equation}
c_m \left ( \mathsf W (\mathcal E, \, \rho^{\star} F_0 )\right ) = [F_0] \, \, \, \text{in} \, \, \, CH_0 (\M),
\end{equation}
and by Grothendieck-Riemann-Roch we have
$$\text{ch} \,  \left ( \mathsf W (\mathcal E, \, \rho^{\star} F_0 ) \right ) = -  \pi_{\star} \left [ \text{ch} \, \e^{\vee} \cdot \rho^{\star} (\text{ch} \, F_0 \cdot (1 + 2 c_X))\right].$$
In particular, if $F\in\M$ is any sheaf such that $c_2 (F)\in CH_0(X)$ is a multiple of $c_X$,
then
\begin{equation} 
c_{\M} = [F] = c_m  \left ( \mathsf W (\mathcal E, \, \rho^{\star} F) \right )
\end{equation}
is manifestly tautological. 

\vskip.1in

\begin{remark}\label{Lagrangian}
As shown in \cite{SYZ}, there exist Lagrangian constant-cycle subvarieties for the special cycle $c_\M$. For any such irreducible subvariety $V \subset \M$ of dimension $n = m/2$, consider the product $V \times M^{n+1}$. The cycles $$\overline \Delta_{01}, \ldots, \overline \Delta_{0,n+1}\,  \in \, CH^m (V \times \M^{n+1})$$ restrict to zero on the fibers of the projection $V \times \M^{n+1} \to V.$ By Theorem 3.1 of \cite{V2}, each of these cycles is supported on the inverse image of a proper closed subset of $V$. Since $V$ has dimension $n$, this ``spreading principle'' implies the vanishing
$$[V] \cdot \overline \Delta_{01} \cdot \overline \Delta_{02} \cdots \overline \Delta_{0,n+1} = 0 \, \, \text{in} \, \, CH_\star (\M \times \M^{n+1}).$$
We now conjecture
\begin{conjecture}
The class of any Lagrangian constant-cycle subvariety for $c_\M$ is in the tautological ring $R_\star (\M).$
\end{conjecture}
\end{remark}

\vskip.2in

\subsection{Vanishing of product cycles}\label{sec:Diagonals}

We now show that the system of product identities of Conjecture 1 leads in turn to a large collection of conjectural Chow vanishings in the self-products $\M \times \M \times \cdots \times \M.$ Among them is the vanishing of O'Grady's modified diagonal cycle. Aligned with our point of view, the modified diagonal cycle is also cast here in product form.

To start, let us single out the complex
\begin{equation}
\mathsf W (\mathcal E, \overline \f) = \ext_{\pi}^{\bullet} (\mathcal E, \, \mathcal F - \rho^{\star} F)  [1] \, \, \, \text{on} \, \, \, \M \times \M,
\end{equation}
where $F$ represents the special zero-cycle, $[F] = c_{\M},$ and $\f, \, \e$ are the universal objects on the first and second factors respectively. 

\comment{The building blocks for its Chern classes are, by Grothendieck-Riemann-Roch, the cycles
$$\alpha_{ij} = \pi_{\star} \left [{\text ch}_i \mathcal E^{\vee} \cdot ({\text ch}_j \overline \f) \cdot \text{td} X \right ].$$}

\noindent
We also set $$\overline \Delta = \Delta - \M \times c_{\M}  = c_m \left (\mathsf W (\mathcal E, \f) \right ) - c_m \left (\mathsf W( \e, \, \rho^{\star} F) \right ) \, \, \, \text{in} \, \, \, CH_m (\M \times \M ).$$ 
Further, in the context of a product $\M \times \,\M^{\ell} \times X$, we let $\e_1, \ldots, \e_{\ell},\,  \f $ be the universal sheaves corresponding to the last $\ell$ factors of $\M$, respectively the first distinguished factor. We label this factor by $0$, and
show 

\medskip

\begin{thmn}[\ref{thm:applications}$^*$]
For any class $\alpha \in CH_\star (\M)$ of codimension $d$ satisfying the inequality $$ d+ \ell > \dim \M,$$
the following three vanishing statements are equivalent.
\begin{enumerate} [label=(\roman*)]
\vskip.2in
\item $\alpha \cdot ch_{i_1} (\overline \f_1) \cdots ch_{i_\ell} (\overline \f_{\ell}) = 0 \in CH_{\star} (\M \times X^{\ell}),$ for all $i_1, \ldots, i_\ell \geq 0.$ Here $\overline \f_s, \, 1 \leq s \leq \ell,$ is the normalized universal sheaf pulled back from $\M$ and the $s$th factor in the product $X^\ell$. 

\vskip.1in
\item $\alpha \cdot c_{i_1} (\mathsf W(\e_1, \overline \f )) \cdot c_{i_2} (\mathsf W (\e_2, \overline \f)) \cdot \cdots \cdot c_{i_{\ell}}(\mathsf W (\e_{\ell}, \overline \f) ) = 0 \in CH_\star \left (\M \times \,\M^{\ell} \right ),$ for all 
$i_1, \ldots, i_\ell >0$. Here the complex $\mathsf W(\e_s, \overline \f ), \, 1\leq s \leq \ell,$ is pulled back from the distinguished factor $\M$ and the $s$th factor in the product $\M^\ell.$
\vskip.1in
\item $\alpha \cdot \overline \Delta_{01} \cdots \overline \Delta_{0, \ell} = 0 \in CH_\star \left (\M \times \,\M^{\ell} \right ),$\label{fromintro}
\end{enumerate}

\vskip.1in
In all three cases, the class $\alpha$ is pulled back to the product from the first distinguished factor $\M$. 
\end{thmn}

\vskip.1in

\begin{remark}
Note that $(i)$ is the vanishing predicted by Conjecture \ref{principal} in case $\alpha$ is tautological. Theorem \ref{thm:applications} is therefore the equivalence of $(i)$ and $(iii)$ for $\alpha \in R_\star (\M)$ and is completely subsumed by the statement above. The vanishing of the modified diagonal cycle, corresponding to $\alpha = 1$, is thus implied by Conjecture \ref{principal}.
\end{remark}

{\noindent \it Proof.} 
We show first that $(i)$ implies $(ii)$. To start, we note that for the complex $$\mathsf W (\mathcal E, \overline \f)\, \, \, \text{on} \, \, \,  \M \times \M,$$ each Chern class $c_k \left (\mathsf W (\mathcal E, \overline \f) \right )$ for $k> 0$ is expressible in terms of pure-degree pieces of the Chern character, and is therefore a sum of products of factors of type  
$$\alpha_{ij} = \pi_{\star} \left [{\text ch}_i \mathcal E^{\vee} \cdot {\text ch}_j \overline \f \cdot \text{td} X \right ] \, \, \, \text{and} \, \, \, \beta_{ij} = \pi_{\star} \left [{\text ch}_i \mathcal E^{\vee} \cdot {\text ch}_j \overline \f \right ].$$
\comment{The same is true for the differences 
$$c_k  \left (\mathsf W (\mathcal E, \f) \right ) - c_k \left (\mathsf W( \e, \, \rho^{\star} F) \right ) =  c_k \left (\mathsf W (\mathcal E, \, \rho^{\star} F + \overline \f) \right ) - c_k \left (\mathsf W( \e, \, \rho^{\star} F) \right ), \, \, \, \text{for} \, \, k >0.$$}
We consider now the larger product 
$$\M \times \M^{\ell}\times X^{\ell}$$
along with a class $\alpha \in CH_{\star} (\M)$ pulled back from the distinguished first factor $\M$,
satisfying $$\text{codim} \, \alpha + \ell > \dim \M.$$
We let $\mathcal F_1, \ldots \mathcal F_{\ell}$ be the universal sheaves pulled back from $\M \times X^{\ell}$, where $\M$ is the distinguished first factor. We also consider the universal sheaves $\e_1, \ldots \e_{\ell}$ on the new factors of $\M$ each paired with a factor of $X$. 

By $(i)$, the vanishing 
$$\alpha \cdot {\text ch}_{j_1} \overline \f_1 \cdots {\text{ch}}_{j_\ell} \overline \f_{\ell} = 0 $$ holds in $CH_{\star} (\M \times \M^{\ell} \times X^\ell)$, pulled back from 
$\M \times X^{\ell}.$
This trivially implies the vanishing of the larger product
$$\alpha\, \cdot \, {\text ch}_{i_1} \e_1^\vee \cdot  {\text ch}_{j_1} \overline \f_1 \cdot  {(\text{td}}\, X_1)^{a_1} \, \cdots  \,{\text ch}_{i_\ell} \e_\ell^\vee \cdot  {\text ch}_{j_\ell} \overline \f_\ell \cdot  {(\text{td}}\, X_\ell)^{a_\ell}  = 0$$ in $CH_{\star} (\M \times \M^{\ell} \times X^\ell),$
for $i_1, j_1, \ldots i_{\ell}, j_{\ell} \geq 0.$ Here the exponents $a_1, \ldots, a_\ell$ are either $0$ or $1$. Pushing forward via the projection $\M \times \M^{\ell}\times X^{\ell} \to \M \times \M^{\ell}$ gives 


\begin{lemma}
Consider the product $\M \times \, \M^{\ell} \, \times X$ and a cycle $\alpha$ on the distinguished factor $\M$, satisfying $\text{codim} \, \alpha + \ell > \dim \M.$ Denote by $\e_1, \ldots \e_{\ell}, \f$ the universal sheaves associated with the last $\ell $ copies of
$\M$, respectively the first one, pulled back to the product $\M \times\,  \M^{\ell} \times \, X. $
Let $\pi: \M \times \M^{\ell} \times X \to \M \times \M^{\ell}$ be the projection. Then
\begin{equation}
\label{basiccor}
\alpha \cdot {\prod}_{k=1}^{\ell} \pi_{\star} \left [ {\text ch}_{i_k} \mathcal E^{\vee}_k \cdot {\text ch}_{j_k} \overline \f  \cdot ({\text{td}} \,X)^{a_k} \right ]
= 0 \, \, \text{in} \, \, CH_{\star} (\M \times \, \M^\ell) ,
\end{equation}
for any  $i_k, j_k \geq 0,$ and $a_k$ taken either $0$ or $1$.
\end{lemma}
As observed earlier, each factor $c_{i_k} \left (\mathsf W (\mathcal E_k, \, \overline \f) \right )$ in the products $(ii)$ of Theorem \ref{thm:applications}$^*$ is a sum of terms each containing a factor of type appearing in \eqref{basiccor} of the lemma, so the vanishings $(ii)$ follow.

\vskip.2in

Notice next that $(ii)$ implies $(iii).$ Indeed, we have in $K$-theory,
$$\mathsf W(\e, \f ) = \mathsf W(\e, \overline \f ) + \mathsf W(\e, \rho^\star F ) \, \, \text{in} \, \, K (\M \times \M),$$
therefore $$c_m (\mathsf W(\e, \f )) = \sum_{i=0}^{m} c_i (\mathsf W(\e, \overline \f )) \cdot c_{m-i} (\mathsf W(\e, \rho^\star F )),$$
and $$\overline \Delta = c_m (\mathsf W(\e, \f )) - c_m (\mathsf W(\e, \rho^\star F)) = \sum_{i=1}^m  c_i (\mathsf W(\e, \overline \f )) \cdot c_{m-i} (\mathsf W(\e, \rho^\star F )).$$
It is thus clear that any term in the expansion of the product $ \overline \Delta_{01} \cdots \overline \Delta_{0, \ell}$ contains a product $c_{i_1} (\mathsf W(\e_1, \overline \f )) \cdot c_{i_2} (\mathsf W (\e_2, \overline \f)) \cdot \cdots \cdot c_{i_{\ell}}(\mathsf W (\e_{\ell}, \overline \f) )$ for some $i_1, \ldots, i_\ell >0.$ Accordingly, the vanishing $(ii)$ implies $(iii)$.

\vskip.2in

Finally, it is easy to see that $(iii)$ implies $(i).$ 
We start with the vanishing
$$\alpha \cdot \overline \Delta_{01} \cdots \overline \Delta_{0, \ell} = 0 \in CH_\star \left (\M \times \,\M^{\ell} \right ),$$ pulled back from $\M \times \, \M^\ell$ to the larger product $\M \times \, \M^\ell \times \, X^\ell.$ Trivally, we also have 
$$\alpha \cdot \overline \Delta_{01} \cdots \overline \Delta_{0, \ell} \cdot \text{ch}_{i_1} \, (\f_1) \cdots \text{ch}_{i_\ell}\, (\f_\ell) = 0 \in CH_\star \left (\M \times \,\M^{\ell} \times \, X^\ell \right ),$$ for any $i_1, \ldots, i_\ell \geq 0.$ Here each $\f_s$ is pulled back from a factor $\M \times X$ in the product $\M^\ell \times X^\ell.$
Pushing forward under the projection $\M \times \,\M^{\ell} \times \, X^\ell \rightarrow \M \times X^\ell$ gives 
$$ \alpha \cdot\text{ch}_{i_1} \, (\overline \f_1) \cdots \text{ch}_{i_\ell}\, (\overline \f_\ell) = 0 \in CH_\star \left (\M \times \, X^\ell \right ),$$ for any $i_1, \ldots, i_\ell \geq 0.$ This concludes the proof of the theorem.
\qed

\vskip.1in

\subsection{Extension of Conjecture \ref{principal}}

We end this section by formulating the following natural extension of our main conjecture. In the context of the product $\M \times X^k \times X^\ell,$ let us index by $\{1, \ldots, \ell \}$ the individual factors in the product $X^\ell$ and by 
$\{\hat 1, \ldots, \hat k \}$ the factors in the product $X^k.$ As usual, $\overline \f_t$ denotes the normalized universal sheaf from the $t$-th factor.

\begin{conj}[\ref{principal}$^*$]
For any integers $k \geq 0$ and $\ell >0 $ consider the product $\M \times \, X^k \times \, X^\ell$ and a tautological class $\alpha \in R^d(\M\times X^k)$ satisfying 
$$d + \ell > \dim ( \M \times X^k).$$
For any  indices $i_1,\ldots,i_\ell\geq 0$, partition $\Omega \sqcup\Theta = \{1,\ldots, \ell\}$, assignment $s:\Theta\to\{\hat 1,\ldots,\hat k\}$
we have
\begin{equation*}
\alpha \cdot \prod_{t\in \Omega}\ch_{i_t}\left(\overline{\mathcal{F}}_t \right)\cdot \prod_{t\in \Theta} \ch_{i_t} ( \mathcal O_{\overline{\Delta}_{s_t,t}} ) =0 \, \, \, \text{in} \, \, \, CH_\star\left(\M\times \, X^k\times \, X^\ell\right).
\end{equation*}
\end{conj}

Observe that Conjecture \ref{principal} is a special case of Conjecture \ref{principal}$^*$, specifically the case $k=0$ and (necessarily) $\Theta=\varnothing$. In $K$-theory, Conjecture \ref{principal}$^*$ predicts the natural generalization of \eqref{conjKTheory}, namely that for any $0\le a\le \ell =\dim\M+2k+1$ and assignment $s:\{a+1,\ldots, \ell\}\to\{\hat 1,\ldots,\hat k\}$, we have:
\begin{align*}
    \overline{\mathcal{F}}_1\cdot\overline{\mathcal{F}}_2\cdot\cdots\cdot\overline{\mathcal{F}}_{a}\cdot\mathcal{O}_{\overline{\Delta}_{s_{a+1},a+1}}\cdot\mathcal{O}_{\overline{\Delta}_{s_{a+2},a+2}}\cdot\cdots\cdot\mathcal{O}_{\overline{\Delta}_{s_{\ell},\ell}}=0.
\end{align*}
In Section \ref{sec: hilb scheme} we will prove the following, which implies Theorem \ref{principal0} in the introduction.
\begin{thmn}[\ref{principal0}$^*$]
Conjecture \ref{principal}$^*$ holds when $\M=X^{[n]}$ is the Hilbert scheme of $n$ points on $X$.
\end{thmn}

\vskip.3in

\section{The product identities for \texorpdfstring{$\M=X^{[n]}$}{M=X\^{}[n]}} 
\label{sec: hilb scheme}

\vskip.3in

The aim of this section is to prove Theorem \ref{principal0}$^*$. We let $\mathcal{I}_n$ denote the ideal sheaf of the universal subscheme $$\mathcal Z_n \subset X^{[n]}\times X$$ and set
$$\overline{\mathcal{I}}_n:=\mathcal{I}_n-\rho^*\mathsf I_n,$$
the rank zero virtual universal sheaf, where $\mathsf I_n$ is the ideal sheaf on $X$ of any subscheme of length $n$ with $c_2 (\mathsf I_n) = n \,c_X.$
We state the theorem explicitly. In the context of the product $X^{[n]} \times X^k \times X^\ell,$ let us index by $\{1, \ldots, \ell \}$ the individual factors in the product $X^\ell$ and by 
$\{\hat 1, \ldots, \hat k \}$ the factors in the product $X^k.$ Further, $\overline{\mathcal I}_n^{(t)}$ denotes the normalized universal sheaf from the $t$-th factor.  We then restate:

\medskip

\begin{thmn}[\ref{principal0}$^*$]
For any integers $n\ge1$, $\ell\geq 1$, $k\geq 0$, consider the product $X^{[n]} \times \, X^k \times \, X^\ell$ and a tautological class $\alpha \in R^d(X^{[n]} \times X^k)$ satisfying 
$$d + \ell >  2n + 2k.$$
For any  indices $i_1,\ldots,i_\ell\geq 0$, partition $\Omega \sqcup\Theta = \{1,\ldots, \ell\}$, assignment $s:\Theta\to\{\hat 1,\ldots,\hat k\}$
we have
\begin{equation*}
\alpha \cdot \prod_{t\in \Omega}\ch_{i_t}\left(\overline{\mathcal{I}}_n^{(t)} \right)\cdot \prod_{t\in \Theta} \ch_{i_t} ( \mathcal O_{\overline{\Delta}_{s_t,t}} ) =0 \, \, \, \text{in} \, \, \, CH_\star\left(X^{[n]}\times \, X^k\times \, X^\ell\right).
\end{equation*}
\end{thmn}



\subsection{Induction preliminaries}

We argue inductively on the number of points using the geometry of the nested Hilbert scheme 
$$X^{[n,n+1]}\subset X^{[n]}\times X^{[n+1]}$$ 
parametrizing pairs $(\xi,\xi')\in X^{[n]}\times X^{[n+1]}$ such that $\xi\subset \xi'$. The inductive technique was first used in
\cite{EGL} to relate top intersections  on $X^{[n+1]}$ and $X^{[n]}\times X$; we now recall its main features. 
Each point $(\xi,\xi')\in X^{[n,n+1]}$ corresponds to an exact sequence
\begin{equation}
\label{ex.seq}
0\to I_{\xi'}\to I_\xi\to\mathcal{O}_{x}\to0,
\end{equation}
leading to projection maps
\begin{equation}
\label{nestedHilbertschemediagram}
\begin{tikzcd}
&X^{[n,n+1]}\arrow[dl, "\phi"'] \arrow [d, "p"] \arrow[dr, "\psi"]&\\
X^{[n]} & X& X^{[n+1]}
\end{tikzcd},
\end{equation}
and globally to an isomorphism
\[X^{[n,n+1]}\cong\P(\mathcal{I}_n)\]
of smooth projective varieties. 

An important role for the geometry of $X^{[n, n+1]}$ is played by the hyperplane line bundle 
$$\mathcal{L}=\mathcal{O}_{\P (\mathcal{I}_n)}(1)$$ with first Chern class 
$c_1 (\mathcal L) = \lambda.$ 
Letting $$\sigma = \phi \times p: X^{[n, n+1]} \to X^{[n]} \times X,$$
we have (cf. \cite{EGL}) 
\begin{equation}
\label{Lemma1.1:EGL}
\sigma_\star\left(\lambda^i\right)=(-1)^ic_i(-\mathcal{I}_n).
\end{equation}

The following fundamental exact sequence on $X^{[n, n+1]} \times X$ relates the universal ideal sheaves and the exceptional line bundle $\mathcal{L}$:
\begin{equation}
\label{ex.seq1}
0\to \psi^{\star}_X \mathcal{I}_{n+1}\to \phi^{\star}_X \mathcal{I}_n \to \pi^{\star}\mathcal{L}\otimes \sigma_X^\star\mathcal{O}_{\Delta}\to 0. 
\end{equation}
Here (and elsewhere in the paper) we use the notation $f_X = f \times \text{id}_X$; the map $\pi:\P(\mathcal{I}_n)\times X\to \P(\mathcal{I}_n)$
is the projection; $\Delta$ denotes the diagonal in $X \times X$, pulled back in \eqref{ex.seq1} via $\sigma_X: X^{[n, n+1]} \times X \to X^{[n]} \times X \times X.$

\vskip.1in

Furthermore, for a vector bundle $F \to X,$ let $F^{[n]}$ denote the associated tautological vector bundle on $X^{[n]},$ 
$$F^{[n]} = \pi_{\star} \left ( \mathcal O_{\mathcal Z_n} \otimes \rho^{\star} (F) \right ).$$
As usual in this text, $\pi$ and $\rho$ are the projections from $X^{[n]} \times X$ to the first and second factors 
respectively. The $K$-theoretic equality
$$\psi^{\star} F^{[n+1]} = \phi^\star F^{[n]} + \mathcal L \cdot p^\star F \,\, \, \text{in} \, \, \, K (X^{[n, n+1]})$$ follows from \eqref{ex.seq1}.
In particular, 
\begin{equation}
\label{boundary}
 \mathcal L = \psi^{\star} \mathcal O^{[n+1]} - \phi^\star \mathcal O^{[n]}  \,\, \, \text{in} \, \, \, K (X^{[n, n+1]}).
\end{equation}

\vskip.2in

The induction in \cite{EGL} only tracks degrees of top codimension classes on the Hilbert scheme $X^{[n]}$. Since our argument involves 
Chow classes of arbitrary codimension, we note the following 

\begin{lemma}
\label{lemmaXn:1}

Let $\alpha$ be a class in $CH_\star(X^{[n+1]} \times X^k).$ Then $$\alpha=0 \iff 
\sigma_\star \psi^{\star} \alpha= \sigma_\star (\lambda \cdot \psi^{\star} \alpha) = 0 \, \, \text{in} \, \, CH_\star(X^{[n]}\times X\times X^k).$$
\end{lemma}

\vskip.1in

\noindent
In the statement of the lemma and also onwards, we abuse notation and denote
\[
\begin{aligned}
\psi&=\psi\times \mathrm{id}_{X^{k}}\colon X^{[n,n+1]}\times X^k\rightarrow X^{[n+1]}\times X^k\\
\sigma&=\sigma\times \mathrm{id}_{X^{k}}\colon X^{[n,n+1]}\times X^k\rightarrow X^{[n]}\times X \times X^k,
\end{aligned}
\]

\vskip.in

\begin{proof}
We use the description of the Chow groups of $$\text {Hilb} = \amalg_{n \geq 0} \,X^{[n]}$$
 in \cite{dC-M}, as well as an aspect of Lehn's formulas in Chow recently established in \cite{MN}. 
Recall first the definition of the Nakajima operators $\mathfrak q_{i}, \mathfrak q_{-i}, i>0.$ For every $i>0$ consider the subscheme
$$X^{[n, n+i]} = \{(I \supset I') \, \text{s.t.} \, I/I' \, \text{is supported at a single point} \, x \in X\} \subset X^{[n]} \times X^{[n+i]},$$
with accompanying maps
$$
\begin{tikzcd}
&X^{[n,n+i]}\arrow[dl, "\phi"'] \arrow [d, "p"] \arrow[dr, "\psi"]&\\
X^{[n]} & X& X^{[n+i]}
\end{tikzcd}
$$
This correspondence defines the operators:
$$\mathfrak q_{\pm i}: \, CH_{\star} (\text{Hilb}) \to CH_{\star} (\text{Hilb} \times X),$$
$$\mathfrak q_{i} = (\psi \times p)_{\star} \circ \phi^{\star}, $$
$$\mathfrak q_{-i} = (\phi \times p)_{\star} \circ \psi^\star, $$
which satisfy the commutation relations in the Heisenberg algebra. 
By composition one can form any string operator $$\mathfrak q_{i_1} \cdots \mathfrak q_{i_\ell}: CH_{\star} (\text{Hilb}) \to CH_{\star} (\text{Hilb} \times X^\ell).$$
Any class $\Gamma \in CH_{\star} (X^\ell)$ then defines an endomorphism of $CH_{\star} (\text{Hilb})$ via
$$\mathfrak q_{i_1} \cdots \mathfrak q_{i_\ell} (\Gamma)=\pi_{1*}(\mathfrak q_{i_1} \cdots \mathfrak q_{i_\ell}\cdot \pi_2^*(\Gamma)),$$
where $\pi_1, \pi_2\colon \text{Hilb} \times X^\ell \rightarrow \text{Hilb},X^\ell$ are the standard projections.

The main theorem of \cite{dC-M} establishes that all Chow classes of the Hilbert scheme arise from Chow classes of symmetric products $X^{\ell}$ through the Nakajima correspondences. To state this precisely, let the vacuum vector $v$ be the generator of $CH_{\star} (X^{[0]}) \simeq \mathbb Q.$ Then 
\begin{equation}
\label{dcm}
CH_{\star} (\text{Hilb}) = \bigoplus_{\substack{ n_1 \geq \ldots \geq n_\ell > 0 \\ \Gamma \in CH_{\star} (X^\ell)^{\text{sym}}}} \mathbb Q \cdot \mathfrak q_{n_1} \cdots \mathfrak q_{n_\ell} (\Gamma) \cdot v,
\end{equation}
where $CH_{\star} (X^\ell)^{\text{sym}} \subset CH_{\star} (X^\ell)$ denotes the subring of classes invariant under transpositions $(ij)$ for which $n_i = n_j$. 
The isomorphism \eqref{dcm} is induced by a correspondence whose transpose gives the inverse map. It follows that for each $n$,
$$\bigoplus_{\substack{n_1 + \cdots  +n_\ell = n \\ n_1 \geq \ldots \geq n_\ell > 0}} \mathfrak {q_{-n_1}} \cdots \mathfrak q_{-n_\ell}    : \,  CH_{\star} (X^{[n]}) \longrightarrow \bigoplus_{\substack{n_1 + \cdots  +n_\ell = n \\ n_1 \geq \ldots \geq n_\ell > 0}} CH_{\star} (X^\ell)$$ is injective. 

The Chow groups of the products $\text{Hilb} \times X^k, \, k>0,$ admit a parallel description, since the inverse of the isomorphism \eqref{dcm} is induced by the transpose correspondence. We thus have 

\begin{equation}
CH_{\star} (\text{Hilb} \times X^k) = \bigoplus_{\substack{n_1 \geq \ldots \geq n_\ell > 0 \\ \Gamma \in CH_{\star} (X^{\ell +k})^{\text{sym}}} } \mathbb Q \cdot \mathfrak q_{n_1} \cdots \mathfrak q_{n_l} (\Gamma) \cdot v.
\end{equation}
(Here $CH_{\star} (X^{\ell +k})^{\text{sym}} \subset CH_{\star} (X^{\ell+k})$ is the subring of classes invariant under transpositions on the first $\ell$ factors for which $n_i = n_j.$)
Correspondingly, for every $n$ and $k$, the map
$$\bigoplus_{\substack{n_1 + \cdots  +n_\ell = n \\ n_1 \geq \ldots \geq n_\ell > 0}} \mathfrak {q_{-n_1}} \cdots \mathfrak q_{-n_\ell}    : \,  CH_{\star} (X^{[n]} \times X^k) \longrightarrow \bigoplus_{\substack{n_1 + \cdots  +n_\ell = n \\ n_1 \geq \ldots \geq n_\ell > 0}} CH_{\star} (X^{\ell +k})$$ is injective.

\vskip.1in

For a class $\alpha \in CH_{\star} (X^{[n+1]} \times X^k)$ we thus have, relevant to the lemma: 
$$\mathfrak q_{-i} (\alpha) = 0 \, \, \text{for all} \, \, i > 0 \implies \alpha = 0.$$

Assume now that $\sigma_\star \psi^{\star} \alpha= \sigma_\star (\lambda \cdot \psi^{\star} \alpha) = 0.$
We note right away that 
$$ \sigma_\star \psi^{\star} = \mathfrak q_{-1} : CH_{\star} \left (X^{[n+1]}  \times X^k \right ) \to CH_\star \left (X^{[n]} \times X \times X^k \right ) .$$ 
Furthermore, recall from \cite[Definition 3.8]{L} that the boundary operator $$\mathfrak \delta: \, CH_\star (X^{[n]}) \to CH_\star (X^{[n]})$$ represents multiplication by the divisor class $c_1 (\mathcal O^{[n]})$ on $X^{[n]}.$ Since via \eqref{boundary} we have
$$\lambda = \psi^{\star} c_1 (\mathcal O^{[n+1]} ) - \phi^{\star} c_1 (\mathcal O^{[n]}),$$
the operation of pulling back via $\psi,$ intersecting with the hyperplane $\lambda$, and pushing forward by $\sigma$, is the commutator $\mathfrak q_{-1}^{(1)}$ of 
$\mathfrak q_{-1}$ with the boundary $\mathfrak \delta$, 
$$ \sigma_\star (\lambda \cdot \psi^{\star}  ) = [\mathfrak \delta, \, \mathfrak q_{-1}] = \mathfrak q_{-1}^{(1)}: \, CH_\star (X^{[n+1]} \times X^k) \to CH_\star (X^{[n]} \times X \times X^k).$$
It is known that the operators $\mathfrak q_{-1}$ and $\mathfrak q_{-1}^{(1)}$ generate all Nakajima lowering operators on the level of Chow, as explained in \cite{MN}. To be precise (cf. \cite{MN} equation (1.12) in Theorem 1.7), we have 
$$[\mathfrak q_{-1}^{(1)}, \, \mathfrak q_{-i}] = i \,\Delta_\star (\mathfrak q_{-i-1}).$$
The left and right side are homomorphisms $CH_{\star} (\text{Hilb} \, \times X^k) \to CH_{\star} (\text{Hilb} \, \times X^2 \times X^k)$ where the last $k$ factors of $X$ are inert for the Chow action. 
Thus $$\mathfrak q_{-1} (\alpha) = \mathfrak q_{-1}^{(1)} (\alpha) = 0 \implies \mathfrak q_{-i} (\alpha) = 0 \, \, \text{for all} \, \, i > 0 \implies \alpha = 0.$$
This ends the proof of the lemma. 
\end{proof}

\vskip.2in

\subsection{Proof of Theorem \ref{principal0}\texorpdfstring{$^*$}{*}} We now complete the inductive argument giving the theorem. 
\subsubsection{The base case \texorpdfstring{$n=1$}{n=1}}\label{sec: base case}

The proof of Theorem \ref{principal0}$^*$ follows in this case from the three identities stated in the introduction, 
\begin{equation}\label{eq1}
\overline \Delta_{01} \cdot \overline \Delta_{02} \cdot \overline \Delta_{03} = 0 \, \, \, \text{in} \, \, \,  CH_2 (X \times X^3)
\end{equation}
\begin{equation}\label{eq2}
D^{(0)}\cdot \overline \Delta_{01} \cdot \overline \Delta_{02} = 0 \, \, \, \text{in} \, \, \,  CH_1 (X \times X^2),
\end{equation}
\begin{equation}\label{eq3}
c_X^{(0)} \cdot \overline \Delta_{01} = 0 \, \, \, \text{in} \, \, \,  CH_0 (X \times X),
\end{equation}
of which the first one is the Beauville--Voisin identity. 
\vskip.1in

\noindent
We have $X^{[1]} = X$ and the universal ideal sheaf is $$\mathcal I_1 = \mathcal I_\Delta \, \, \, \text{on} \, \, \,  X \times X.$$
A tautological class $\alpha \in R^d (X \times X^k)$ is a polynomial in pullbacks of diagonals, divisors, and special cycles $c_X$ from various
factors of the product $X^{k+1}.$ Thus $\alpha$ is necessarily a linear combination of subvarieties of $X^{k+1}$ of type 
\begin{equation}
\label{taut1}
c_X^{(1)} \times \cdots \times c_X^{(a)} \times D_1^{(a+1)} \times \cdots \times D_b^{(a+b)} \times X^{(a+b+1)} \times \cdots \times X^{(a+b+c)} \subset X^{k+1}
\end{equation}
where the embedding in $X^{k+1}$ is by diagonals (and up to ordering of the factors). 
Here $a + b + c \leq k+1$ and $$\dim \alpha = b + 2c.$$

We now assume $\alpha$ is of the form \eqref{taut1}. Indexing the first copy of $X$ by $0$, we have $$\overline{\mathcal{I}}_{1}^{(t)}= - \mathcal{O}_{\overline{\Delta}_{0,t}}$$ in $K$-theory.  We seek to establish the vanishing
$$\alpha \cdot \prod_{t\in \Omega}\ch_{i_t}\left( \mathcal{O}_{\overline {\Delta}_{0,t}} \right)\cdot \prod_{t\in \Theta} \ch_{i_t} ( \mathcal O_{\overline{\Delta}_{s_t,t}} ) =0 \, \, \, \text{in} \, \, \, CH_\star \left( X \times \, X^k\times \, X^\ell \right),$$
whenever $\ell > \dim \alpha.$ Here $t$ runs through $\{1, \ldots, \ell \}$ and $\Omega \sqcup \Theta = \{1, \ldots, \ell \}$. The Chern character degrees $i_t$ are arbitrary. 

Noting now that in $X \times X$ we have 
$$ \ch(\mathcal{O}_{\overline \Delta} ) = \overline{\Delta} -2 \, c_X \times c_X$$ and the cycle $c_X \times c_X$ is a rational multiple of $\overline{\Delta}^2$ (cf. \cite{BV}), it is enough to show the vanishing
\begin{equation}
\label{finalproduct}
\alpha \cdot \prod_{t\in \Omega}\overline {\Delta}_{0,t} \cdot \prod_{t\in \Theta} \overline{\Delta}_{s_t,t} =0 \, \, \, \text{in} \, \, \, CH_\star \left( X \times \, X^k\times \, X^\ell \right),
\end{equation}
 whenever $\ell > \dim \alpha = b + 2c.$ Since $\alpha$ is of the form \eqref{taut1}, this inequality guarantees that a factor of $c_X$ receives a matching normalized diagonal, or a factor of $D$ receives two matching normalized diagonals, or a factor of $X$ receives three matching normalized diagonals. (Here "matching" means that the normalized diagonal shares an index with the class in question.) The fundamental identities \eqref{eq1}, \eqref{eq2}, \eqref{eq3} therefore ensure that the product \eqref{finalproduct} vanishes. 
 \qed

\vskip.2in

\subsubsection{The induction step}
Let $\alpha\in R^d (X^{[n+1]}\times X^k ).$ We want to show that for every $\ell$ satisfying 
$$\ell > \dim \alpha,$$ and any indices $i_1, \ldots, i_\ell \geq 0,$ the class 
\begin{equation}
\label{eq:EqnNPlus1}
  \gamma :=  \alpha\cdot \prod_{t\in\Omega }\ch_{i_t}(\overline{\mathcal{I}}_{n+1}^{(t)})\cdot\prod_{t\in\Theta}\ch_{i_t} (\mathcal O_{\overline{\Delta}_{s_t,t}} )=0 \, \, \text{in} \, \, CH_\star (X^{[n+1]}\times X^k \times X^\ell).
\end{equation}
According to Lemma \ref{lemmaXn:1} it suffices to show $$\sigma_{\star} \psi^\star \gamma = \sigma_\star (\lambda \cdot \psi^\star \gamma) = 0 \, \, \text{in} \, \,  CH_\star (X^{[n]}\times X^{k+1} \times X^\ell).$$

To start, note that as a tautological class, $\alpha$ is a polynomial in classes 
$$\beta_j, \hskip.2in \Delta_{s,s'}, \hskip.2in D^{(s)}, \hskip.2in c_{X}^{(s)}, \hskip.2in\hbox{and}\hskip.2in \text{ch}_i(\mathcal{I}_{n+1}^{(s)}),$$
with $\beta_j \in R^\star(X^{[n+1]})$ and $s\in \{\hat 1,\ldots,\hat k\}.$ Recalling the exact sequence \eqref{ex.seq1},
$$0\to \psi^{\star}_X \mathcal{I}_{n+1}\to \phi^{\star}_X \mathcal{I}_n \to \pi^{\star}\mathcal{L}\otimes \sigma_X^\star\mathcal{O}_{\Delta}\to 0 \, \, \text{on} \, \, X^{[n, n+1]} \times X$$
it follows that the pullback $\psi^\star \alpha$ is  of the form
\begin{equation}
\label{tautpullback}
\psi^\star \alpha = \sum_{j=0}^d \alpha_{d-j} \cdot \lambda^j \in CH_\star (X^{[n, n+1]}\times X^k),
\end{equation}
 where $$\alpha_{d-j} \in R^{d-j} (X^{[n]} \times X^{k+1})$$
are pulled back under $\sigma \times \text{id}_{X^k}: X^{[n, n+1]} \times X^k \to X^{[n]} \times X^{k+1}.$ (Here we suppressed the pullback from the notation.)

\vskip.1in

Furthermore, the fundamental exact sequence \eqref{ex.seq1} gives immediately the $K$-theoretic equality
$$\overline{\mathcal I}_{n+1}^{(t)} = \overline{\mathcal I}_{n}^{(t)} - \mathcal L \cdot \mathcal O_{\overline{\Delta}_{0, t}} - (\mathcal L - 1) \cdot \mathcal O_{c_X^{(t)}} \, \, \, \, \text{in } \, \, K (X^{[n, n+1]} \times X^\ell).$$
Here $0$ denotes the factor of $X$ which $X^{[n, n+1]}$ maps to under $p: X^{[n, n+1]} \to X$; we suppressed the pullbacks from the notation; as usual $t$ indexes the factors in $X^\ell.$

Thus the class $\psi^\star \gamma \in CH_\star (X^{[n, n+1]}\times X^k \times X^\ell)$ is a linear combination of terms of the form
\begin{equation}
\label{monomial}
\psi^\star \alpha \cdot \prod_{t\in\Omega_1}\ch_{i_t}(\overline{\mathcal{I}}_n^{(t)}) \cdot  \prod_{t\in\Omega_2} \ch_{i_t} (\mathcal L \otimes \mathcal O_{\overline{\Delta}_{0, t}} )\cdot \prod_{t\in\Omega_3} \text{ch}_{i_t} ((\mathcal L -1)\otimes \mathcal O_{c_X^{(t)}}) \cdot      \prod_{t\in\Theta}\ch_{i_t} (\mathcal O_{\overline{\Delta}_{s_t,t}} )
\end{equation}
Here $\Omega_1 \sqcup \Omega_2 \sqcup \Omega_3 \sqcup \Theta = \{1, \ldots, \ell \},$ the indexing set for factors in the product $X^\ell,$ and $\psi^\star \alpha$ is the codimension $d$ class given by \eqref{tautpullback}. 
The expression \eqref{monomial} is a polynomial in $\lambda$ with coefficients pulled back from $R_\star (X^{[n]} \times X^{k+1} \times X^\ell).$
Noting now that $$\text{ch} ((\mathcal L -1)\otimes \mathcal O_{c_X^{(t)}}) = \lambda \cdot p (\lambda) \cdot c_X^{(t)} \, \, \, \text{for}  \, \, p \in \mathbb Q[\lambda],$$ 
we conclude from  \eqref{monomial} that the classes $\psi^\star \gamma, \, \lambda \cdot \psi^\star \gamma \in CH_\star (X^{[n, n+1]}\times X^k \times X^\ell)$
are linear combinations of terms of type
 $$\tilde \alpha \cdot \prod_{t\in\Omega_1}\ch_{i_t}(\overline{\mathcal{I}}_n^{(t)}) \cdot  \prod_{t\in\Omega_2} \ch_{i_t} (\mathcal O_{\overline{\Delta}_{0, t}} )\cdot \prod_{t\in\Omega_3} c_X^{(t)} \cdot      \prod_{t\in\Theta}\ch_{i_t} (\mathcal O_{\overline{\Delta}_{s_t,t}}),$$
where the quadruple product is now pulled back from $R_\star (X^{[n]} \times X^{k+1} \times X^\ell).$ The class $\tilde \alpha$ is a polynomial in $\lambda$ with coefficients in $R_\star (X^{[n]} \times X^{k+1}),$ and
$$\text{codim} \,  \tilde{\alpha} \geq d + \omega, \, \, \text{where} \, \, \omega = |\Omega_3|.$$
Correspondingly, since powers of $\lambda$ push forward to tautological classes (cf. \eqref{Lemma1.1:EGL}), the pushforwards $\sigma_\star \psi^\star \gamma$ and $\sigma_\star (\lambda\cdot \psi^\star \gamma)$ are linear combinations of terms of the form
$$\beta \cdot \prod_{t\in\Omega_1}\ch_{i_t}(\overline{\mathcal{I}}_n^{(t)}) \cdot  \prod_{t\in\Omega_2} \ch_{i_t} (\mathcal O_{\overline{\Delta}_{0, t}} )\cdot \prod_{t\in\Omega_3} c_X^{(t)} \cdot      \prod_{t\in\Theta}\ch_{i_t} (\mathcal O_{\overline{\Delta}_{s_t,t}}),$$
for $\beta \in R_\star (X^{[n]} \times X^{k+1})$ satisfying $$d' : = \text{codim} \, \beta \geq d + \omega.$$ 
Setting $\ell' = \ell - \omega$ and omitting the $X$ factors which carry a class $c_X,$ we note that the product 
$$\beta \cdot \prod_{t\in\Omega_1}\ch_{i_t}(\overline{\mathcal{I}}_n^{(t)}) \cdot  \prod_{t\in\Omega_2} \ch_{i_t} (\mathcal O_{\overline{\Delta}_{0, t}} ) \cdot      \prod_{t\in\Theta}\ch_{i_t} (\mathcal O_{\overline{\Delta}_{s_t,t}}) = 0 \in R_\star (
X^{[n]} \times X^{k+1} \times X^{\ell'} )$$ by the induction hypothesis, since
$d' + \ell' \geq d + \ell > 2n + 2k + 2.$ \qed

\section{The Chow filtration}
\label{split_filtration}

\vskip.2in

In this section we examine the filtration on $CH_\star (\M)$ proposed in the introduction. The filtration is natural to the product point of view of this paper. Within its framework, all dimensions 
are treated uniformly. For a fixed codimension $d$ and $0 \leq i \leq d$ we set 
\begin{equation}
\label{filtr}
\S_i (CH^d (\M)) =  \{ \alpha \, \, \text{with} \, \, \alpha \cdot \overline \Delta_{01} \cdots \overline \Delta_{0, \,m-d+i+1} = 0 \subset CH_{\star} (\M \times \,\M^{m-d+i+1})\}.
\end{equation}
Here $\alpha$ is pulled back to the product from the first factor. 
We have $$\S_0 \subset \S_1 \subset \S_2 \subset \ldots.$$ On general grounds (cf. Corollary 1.6 and Proposition 2.2 in \cite{V3}), the filtration terminates eventually. A precise bound is predicted by the conjectural vanishing \eqref{moddiag},
$${\overline {\Delta}}_{01} \cdots \overline \Delta_{0,m+1} = 0 \, \, \, \text{in} \, \, \, CH_m (\M\, \times \, \M^{m+1}),$$
which is the codimension zero case of Conjecture \ref{principal}. This identity would ensure that $\S_d (CH^d (\M)) = CH^d (\M).$ As we will see immediately however, for certain cases (e.g., dimension zero) one can establish the expected termination threshold directly. 

\vskip.1in

Theorem \ref{thm:applications}$^*$ is a useful tool for analyzing $\S_\bullet$:  for a class $\alpha \in CH_\star (\M),$ it asserts
 $$\alpha \cdot \overline \Delta_{01} \cdots \overline \Delta_{0, \ell} = 0 \in CH_\star \left (\M \times \,\M^{\ell} \right ) \iff $$
 $$ \alpha \cdot ch_{i_1} (\overline \f_1) \cdots ch_{i_\ell} (\overline \f_{\ell}) = 0 \in CH_{\star} (\M \times X^{\ell}), \, \, \, \text{for all} \, \, \, i_1, \ldots, i_\ell \geq 0.$$

\vskip.2in

\subsection{The filtration for zero-cycles.}  We examine the filtration first for zero-cycles, when $d = m$. In this case we simply have
\begin{equation}
\label{filtr-zero}
\S_i (CH_0 (\M)) =  \{ \alpha \, \, \text{with} \, \, \alpha \cdot \overline \Delta_{01} \cdots \overline \Delta_{0, \,i+1} = 0 \subset CH_{\star} (\M \times \,\M^{i+1})\}.
\end{equation}

For a cycle $\eta \in CH_\star (\M)$, we use the notation
 $$\eta^{\boxtimes \ell} = p_1^\star \eta \cdot p_2^\star \eta \cdots p_\ell^\star \eta \in CH_* (\M^\ell),$$ where $p_j : \M^\ell \to \M,\, \, 1\leq j \leq \ell,$ is the projection on the $j$th factor.
By definition then, for a point class $[F]$ with $F \in \M$, we have 
$$[F] \in \S_i(CH_0 (\M)) \, \, \iff \, \, ([F] - c_\M)^{\boxtimes( i+1)} = 0 \in CH_0 (\M^{i+1}).$$
For a zero-cycle $\alpha \in CH_0 (\M)$, the vanishings 
$$ \alpha \cdot ch_{i_1} (\overline \f_1) \cdots ch_{i_\ell} (\overline \f_{\ell}) = 0 \in CH_{\star} (\M \times X^{\ell}) \, \, \, \text{for} \, \, \, i_1, \ldots, i_\ell \geq 0$$ hold trivially unless $$i_1 = i_2 = \cdots =i_\ell = 2.$$ In the latter case, we can equivalently replace the second Chern character by the second Chern class. We conclude that for the zero-cycle $\alpha = [F],$ Theorem \ref{thm:applications}$^*$ asserts the equivalence
$$
([F] - c_\M)^{\boxtimes( i+1)} = 0\,  \in \, CH_0 (\M^{i+1}) \iff  (c_2 (F) - k c_X)^{\boxtimes (i+1)} =0 \, \in \, CH_0 (X^{i+1}),$$
where 
$$k = \int_X c_2 (F) \in \mathbb Z$$ is the degree of the second Chern class for sheaves with Mukai vector $v = v (F).$ Thus 
\begin{equation}
[F] \in \S_i  (CH_0 (\M))  \iff  (c_2 (F) - k c_X)^{\boxtimes (i+1)} =0 \, \in \, CH_0 (X^{i+1}).
\label{filtration-zero}
    \end{equation}
\vskip.1in

\subsection{Comparison with the Voisin and Shen-Yin-Zhao filtrations on \texorpdfstring{$CH_0 (\M)$}{CH\_{}0(M)}}

\vskip.2in

The filtration $\S_\bullet$ is closely related to the following two filtrations previously studied in the context of zero-cycles on $\M$. For the rest of the paper we set $$m = \dim \M = 2n.$$ Following \cite{V1}, we let 
\begin{equation}\S^V_i = \, \text{subgroup of } \, CH_0 (\M) \, \text{generated by point classes}\, \,  [F], \, F \in \M, \, \, \end{equation} $$\text{with} \, \, \dim O_F\geq n-i. $$
\vskip.1in
\noindent
Here $O_F$ denotes the rational equivalence orbit of a point $F \in \M$, 
$$O_F=\left\{E\in \M\mid \left[E\right]=\left[F\right]\hbox{ in }CH_0(\M)\right\},$$
a countable union of subvarieties of $\M$. The dimension of the orbit is the largest occurring dimension among these subvarieties. 

\vskip.1in

A second filtration was introduced in a general moduli context in \cite{SYZ}, in relation to O'Grady's \cite{OG2} classic filtration on $CH_0 (X).$ We let:
\begin{equation}\S_i^{SYZ} =  \, \text{subgroup of} \, CH_0 (\M) \, \text{generated by point classes}\, \,  [F], \, F \in \M, \end{equation}  $$\text{with} \, \, c_2 (F) = Z_i + (k-i) c_X \in CH_0 (X), \,
\text{where}\, \, Z_i \,\, \text{is effective of degree} \, i.$$ 

\vskip.1in

It is known \cite{V1} that these two filtrations agree for the Hilbert scheme of points, $$\S^V_\bullet (CH_0 (X^{[n]})) = \S^{SYZ}_\bullet (CH_0 (X^{[n]})).$$ In a general setting, for a moduli space $\M$ with arbitrary primitive Mukai vector and generic polarization, it is known \cite{SYZ} that
$$ \S^{SYZ}_\bullet (CH_0 (\M)) \subset  \S^{V}_\bullet (CH_0 (\M)).$$
We record the following easy lemma, providing a connection with the filtration \eqref{filtr-zero}. 

\begin{lemma} \label{lem:FiltrationInclusion} We have 
$\S_i^{SYZ}(CH_0 (\M))\subset \S_i (CH_0 (\M)),$ where $\M$ parametrizes stable sheaves with an arbitrary primitive Mukai vector $v$. \end{lemma}
\begin{proof}
   \noindent
  The fundamental Beauville-Voisin identity 
$$\overline{\Delta}_{01}\cdot\overline{\Delta}_{02}\cdot\overline{\Delta}_{03}=0 \in CH_2 (X^4)$$
gives, for any point $x \in X$:
$$0 = {p_{23}}_\star \left ( \overline{\Delta}_{01}\cdot\overline{\Delta}_{02}\cdot\overline{\Delta}_{03} \cdot p_1^\star[x] \right ) = ([x] - c_X)^{\boxtimes 2} \in CH_0 (X^2).$$
For any collection of points $x_1, \ldots, x_i \in X$ we therefore have
$$ \left ( \left[x_1\right]+\ldots+\left[x_i\right]-ic_X \right )^{\boxtimes (i+1)} = 0 \in CH_0 (X^{i+1}).$$
Recall now that $\S_i^{SYZ}$ is generated by point classes $[F]$ for $F \in \M$ satisfying
  $$c_2(F)-kc_X=\left[x_1\right]+\ldots+\left[x_i\right]-ic_X \in CH_0 (X),$$ for a collection of points $x_1, \ldots, x_i \in X.$ 
For any such $F$ we thus have
$$\left (c_2(F)-k c_X\right) ^{\boxtimes (i+1)} = 0 \in CH_0 (X^{i+1}).$$
By \eqref{filtration-zero}, this means $[F] \in \S_i.$

\end{proof}

\begin{remark}
 It is known (cf \cite{SYZ}, \cite{OG2}, \cite{V5}) that the filtration $\S_\bullet^{SYZ}$ terminates after $n$ steps, $\S_n^{SYZ} = CH_0 (\M).$  By the lemma, so does the filtration $\S_\bullet$ on $CH_0 (\M)$, and we have
\begin{equation}
\label{nterms}
[F] \cdot \overline \Delta_{01} \cdot \overline \Delta_{02} \cdots \overline \Delta_{0, n+1} = 0 \in CH_\star (\M \times \M^{n+1}), \, \, \text{for all} \, \, F \in \M.
\end{equation}
The subtler inclusion $ \S_i^{V}(CH_0 (\M)) \subset  \S_i(CH_0 (\M))$ connects to properties of the filtration for higher-dimensional cycles, as we note in the next subsection. 
\end{remark}

\begin{remark}
It is useful to single out 
the subgroup $$\S_i^{\sm} (CH_0 (\M)) \subset \S_i (CH_0 (\M))$$ generated by point classes $[F] \in \S_i (CH_0 (\M)),$ for $F \in \M.$ Via \eqref{filtration-zero}, the equality $\S_i^{SYZ}(CH_0 (\M)) =\S_i^{\sm}(CH_0 (\M))$ is implied by the following conjectural equivalence noted in \cite{SYZ}:
for a degree-zero cycle $\xi \in CH_0 (X),$ $$\xi = [x_1] + \cdots + [x_i] - i c_X \, \, \text{for a collection} \,\, x_1, \ldots, x_i \in X \iff \xi^{\boxtimes (i+1)} = 0 \in CH_0 (X^{i+1}).$$The obvious direction in this equivalence was already used in the lemma to show $ \S_i^{SYZ}(CH_0 (\M)) \subset \S_i^{\sm}(CH_0 (\M)).$ \end{remark}

\subsection{The filtration for higher-dimensional cycles.}

\vskip.2in

Understanding the filtration $\S_\bullet$ in dimension zero is interestingly tied with properties of $\S_\bullet$ for cycles of higher dimension. 
We propose here the following generalization of identity \eqref{nterms}.

\begin{conjecture}
Let $V \subset \M$ be a constant-cycle subvariety. Then 
\begin{equation}
\label{constant}
[V] \cdot \overline \Delta_{01} \cdot \overline \Delta_{02} \cdots \overline \Delta_{0, n+1} = 0 \in CH_\star (\M \times \M^{n+1}).  
\end{equation}
Here $[V]$ is pulled back as usual from the first $\M$ factor. 

\end{conjecture}

The vanishing is conjectured to hold for a product of $n+1$ normalized diagonals regardless of the dimension of the constant-cycle subvariety $V$. Thus when $[V] \in CH^{m-i} (\M)$, the conjecture positions the class of $V$ which has dimension $i$ in the $n-i$th piece of its Chow group,
$$[V] \in \S_{n-i} (CH^{2n-i} (\M)).$$ This matches an expectation formulated in \cite{V1} (see Proposal (**) of Section 3 therein) regarding the placement of constant-cycle subvarieties in a hypothetical filtration on all of $CH_\star (\M)$ which would split the Bloch-Beilinson filtration. 
Let us now observe
\begin{lemma}
The vanishing \eqref{constant} holds if and only if $\S^V_\bullet (CH_0 (\M)) \subset \S_\bullet(CH_0 (\M)).$
\end{lemma}

\begin{proof}
We assume first that $\S^V_\bullet (CH_0 (\M)) \subset \S_\bullet (CH_0 (\M))$ and will deduce \eqref{constant}. Let $V = V_F$ be a constant-cycle subvariety for $F \in \M$ of dimension $n-i$, so $[F] \in \S_i^V.$ By assumption we have 
\begin{equation}
([F] -c_\M)^{\boxtimes (i+1)}=0 \in CH_0 (\M^{i+1}),
\label{boxf}
\end{equation}
 since $[F] \in \S_i.$
We set $$\overline \Delta^F = \Delta - \M \times [F] \in CH_m (\M \times \M),$$ thus
$$ \overline \Delta = \overline \Delta^F + \M \times ([F] -c_\M).$$
The spreading principle gives
\begin{equation}
\label{spreadvf}
[V_F] \cdot \overline \Delta_{01}^F\cdot \overline \Delta_{02}^F\cdots \overline \Delta^F_{0,n-i+1} = 0 \in CH_\star (\M \times \M^{n-i+1}).
\end{equation}
Expanding, the two identities \eqref{boxf} and \eqref{spreadvf} now imply that 
$$
[V_F] \cdot \overline \Delta_{01}  \cdots \overline \Delta_{0, n+1} = [V_F] \cdot \left (\overline \Delta_{01}^F + ([F] -c_\M)^{(1)}\right ) \cdots \left ( \overline \Delta^F_{0,n+1} + ([F] -c_\M)^{(n+1)} \right ) =0.
$$
(Here $\eta^{(i)}$ denotes the pullback of $\eta \in CH_\star (\M)$ from the $i$th factor to the product $\M^{n+1}.$)

\vskip.1in

Conversely, let us assume identity \eqref{constant} holds, and let $[F] \in \S_i^V,$ for $F \in \M$. Let $V_F$ be a constant-cycle subvariety for $F$ of dimension $j \geq n-i.$ Let $D$ be an ample divisor on $\M$. We have 
$$[V_F] \cdot D^{j} = a \,[F] \in CH_0 (\M),$$ for a positive number $a$. 
By assumption, 
$$[V_F] \cdot \overline \Delta_{01} \cdot \overline \Delta_{02} \cdots \overline \Delta_{0, n+1} \cdot D^{(1)} \cdot D^{(2)} \cdots D^{(j)} = 0 \in CH_0 (\M \times \M^{n+1}).$$
Pushing this cycle forward to the last $n-j+1$ factors of the product $\M^{n+1}$ we obtain
$$a \cdot ([F] -c_\M)^{\boxtimes (n-j+1)}=0 \in CH_0 (\M^{n-j+1}).$$ 
As $j \geq n-i,$ we have $n-j+1 \leq i+1,$ therefore
$$([F] -c_\M)^{\boxtimes (i+1)}=0 \in CH_0 (\M^{i+1})$$ 
and $[F] \in \S_i.$ 
\end{proof}

\begin{remark} 
Let us note that in turn, the inclusion $\S_\bullet^{\sm} \subset \S_\bullet^V$ is implied by the following statement: for any point $F\in \M$ and any constant-cycle subvariety $V_F$ of maximal dimension in the orbit $O_F$, we have 
\begin{equation}
\label{constcycleineq}
[V_F] \cdot  \overline \Delta_{01} \cdot \overline \Delta_{02} \cdots \overline \Delta_{0, n} \neq 0 \in CH_\star (\M \times \M^{n}).
\end{equation}
Indeed, let 
 $[F] \in \S_i,$ therefore $([F] - c_\M)^{\boxtimes (i+1)} = 0 \in CH_0 (\M^{i+1}).$ Let $V_F$ be a maximal constant-cycle subvariety for $F$.  Writing $ \overline \Delta_{0i} = \overline \Delta_{0i}^F + ([F] -c_\M)^{(i)}$ and expanding the product, we get 
\begin{eqnarray*} [V_F] \cdot  \overline \Delta_{01} \cdot \overline \Delta_{02} \cdots \overline \Delta_{0, n} \neq 0 &\implies&  [V_F]  \cdot  \overline \Delta_{01}^F \cdot \overline \Delta_{02}^F \cdots \overline \Delta_{0, n-i}^F \neq 0 \in CH_\star (\M \times \M^{n-i})\\ &\implies & \dim V_F \geq n-i  \implies [F] \in \S_i^V.
\end{eqnarray*}
\end{remark}

\end{document}